\documentclass{article}
\usepackage{amsmath, amsthm, amssymb}

\def\C{\mathbb C}  \def\N{\mathbb N} 
\def\R{\mathbb R}  \def\Z{\mathbb Z}  
\def\D{\mathbb D}

\def\F{\mathcal F}  
\def\G{\Gamma} 
\def\PSL{\operatorname{PSU}(1,1)} 
\def\Homeo{\operatorname{Homeo}_{+}(S^1)} 
\def\tHomeo{\widetilde{\operatorname{Homeo}}_{+}(S^1)} 
\def\hyp{{\rm hyp}}
\def\vol{\operatorname{vol}} 
\def\eu{\operatorname{eu}}
\def\hol{\operatorname{hol}}
\def\m{\mathfrak{m}}
\def\pr{\operatorname{pr}}

\numberwithin{equation}{section}

\title{Harmonic measures and rigidity for surface group actions on the circle}
\author{
  Adachi, Masanori\\
  \texttt{adachi.masanori@shizuoka.ac.jp}
  \and
  Matsuda, Yoshifumi\\
  \texttt{ymatsuda@math.aoyama.ac.jp}
  \and
  Nozawa, Hiraku\\
  \texttt{hnozawa@fc.ritsumei.ac.jp}
}

\date{\today}

\newtheorem{thm}{Theorem}[section]
\newtheorem{cor}[thm]{Corollary}
\newtheorem{lem}[thm]{Lemma}
\newtheorem{prop}[thm]{Proposition}
\theoremstyle{definition}
\newtheorem{defn}[thm]{Definition}
\theoremstyle{remark}
\newtheorem{rem}[thm]{Remark}
\newtheorem{claim}[thm]{Claim}

\begin{document}

\maketitle

\begin{abstract}
We study rigidity properties of actions of a torsion-free lattice of $\PSL$ on the circle $S^1$. We follow the approaches of Frankel and Thurston proposed in preprints via foliated harmonic measures on the suspension bundles. Our main results are a curvature estimate and a Gauss--Bonnet formula for the $S^1$ connection obtained by taking the average of the flat connection with respect to a harmonic measure. As consequences, we give a precise description of the harmonic measure on suspension foliations with maximal Euler number and an alternative proof of rigidity theorems of Matsumoto and Burger--Iozzi--Wienhard.
\end{abstract}

\section{Introduction}

Let $\Sigma$ be an orientable hyperbolic surface of finite type. The $\pi_1(\Sigma)$-actions on the circle $S^1$ have rigidity properties governed by the Euler class.
For closed $\Sigma$, the Euler class is bounded by the Milnor--Wood inequality \cite{Milnor,Wood} and Matsumoto \cite{Matsumoto1} proved that if the equality holds for a $\pi_1(\Sigma)$-action in the Milnor--Wood inequality, then the action is semiconjugate to a Fuchsian action. Burger--Iozzi--Wienhard \cite[Section 4.5]{BIW} generalized the Milnor--Wood inequality and Matsumoto's theorem to the case where $\Sigma$ has finite volume and cusps. They defined and used the following generalization of the Euler number: Let $\D$ be the Poincar\'e disk. For a torsion-free lattice $\G$ in $\PSL=\operatorname{Isom}_+(\D)$ and a homomorphism $\rho : \G \to \Homeo$, the Euler number $e(\rho) \in \R$ is defined based on the bounded cohomology. In the case where $\G$ is not uniform, they showed that it is expressed in terms of the translation number of any homomorphism lifting $\tilde\rho : \G \to \tHomeo$ of $\rho$ to the universal cover group $\tHomeo$ of $\Homeo$: we have 
\[
e(\rho) = -\sum_{i=1}^{m} \tau (\tilde\rho (c_i)),
\]
where $\tau : \tHomeo \to \R$ is the translation number and $c_1, \dots, c_m$ are curves which go around each of cusps of the surface $\Sigma := \G \backslash \D$.

In this article, we study the rigidity of surface group actions on $S^1$ via foliated harmonic measures. We are going to follow an approach proposed in the unpublished article by Frankel \cite{Frankel} and the unfinished paper by Thurston \cite{Thurston}  (cf.\ Calegari \cite[Example 4.6]{Calegari}). During this course, we will address some missing details in these previous works concerning regularity issues (see Lemmas \ref{lem:bilipschitz} and \ref{lem:curvature}). 
We will consider a specifically chosen $S^1$-connection obtained by taking the average of the flat connection on the suspension bundle with respect to a harmonic measure. Our main results concern with this connection: a curvature estimate, which we prove following the master thesis of the first author \cite{Adachi1}, and a Gauss--Bonnet formula (Theorem \ref{thm:GB}). These results led us to give a precise description of the harmonic measure with maximal Euler number (Theorem \ref{thm:Poi}) in addition to an alternative proof of the results of Matsumoto and Burger--Iozzi--Wienhard (Corollary \ref{cor:mat}).

To state our main result, we briefly explain a construction due to Thurston \cite{Thurston} of the specifically chosen $S^1$-connection on the suspension bundle of a given action $\rho : \G \to \Homeo$.
Consider the suspension bundle $\pi \colon \Sigma \times_{\rho} S^1 \to \Sigma$ of $\rho$ equipped with the horizontal foliation $\F$.
Let $\mu$ be a harmonic measure on $\F$, which exists from the result of Alvarez \cite{Alvarez}. We normalize and disintegrate $\mu$ along the fibers of $\pi$ to obtain a family of measures $\{ \mu_z \}_{z \in \Sigma}$ on the fibers with total measure $2\pi$. 
When $\rho(\Gamma)$ has no finite orbit in $S^1$, measures $\mu_z$'s are non-atomic.
Moreover, we may assume that $\mu_z$'s have full support on the fibers by collapsing the complement of the support of $\mu$ by a semiconjugacy. 
We may regard $\Sigma \times_{\rho} S^1$ as a principal $S^1$-bundle whose principal $S^1$-action $\{\phi_{t}\}$ preserves
$\mu_z$ on each fiber. 
Let $\omega$ be the flat connection form defining $\F$. 
By taking the average of $\omega$ under the principal $S^1$-action, we obtain an $S^1$-connection of $\Sigma \times_{\rho} S^1$, 
\[
\bar\omega := \frac{1}{2\pi}\int_0^{2\pi} \phi^*_t\omega\, dt,
\]
which we call the $S^1$-\textit{connection form associated with} $\mu$. It should be noted that, after collapsing the complement of the support of $\mu$, the foliation $\F$ is only transversely Lipschitz and $\omega$ is continuous (see Section \ref{subsect:smooth-str}). Then $\bar{\omega}$ is only continuous in general and its curvature is considered as in the sense of Definition \ref{def:conconn}.

Now we state our main result.
\begin{thm}\label{thm:GB}
Let $\G$ be a torsion-free lattice in $\PSL$ and $\rho : \G \to \Homeo$ a homomorphism. Assume that $\rho (\Gamma)$ has no finite orbit in $S^1$. For a harmonic measure $\mu$ on the suspension foliation of $\rho$, let $\bar\omega$ be the $S^1$-connection form associated with $\mu$. 
Then, $\bar\omega$ has curvature of the form $K \vol$ such that $\lvert K(z) \rvert \leq 1$ a.e.\ $z$, and we have
\[
e(\rho) = \frac{1}{2\pi} \int_{\Sigma} K(z) \vol(dz),
\]
where $\vol$ denotes the hyperbolic volume form on $\Sigma := \G \backslash \D$.
\end{thm}

In the case where $\G$ is a uniform lattice, the Gauss--Bonnet formula is classical and is not new. The main point of this result is the estimate of the curvature and the behavior of the connection $\bar\omega$ near the cusps of $\Sigma$, which is proved based on Harnack's inequality and the isoperimetric inequality.  We remark that there is a similar curvature estimate in \cite[\S 6]{BCG}. 

As a corollary, we obtain an alternative proof of the generalization of the Milnor--Wood inequality and Matsumoto's rigidity theorem to torsion-free lattices due to \cite{BIW}.

\begin{cor}\label{cor:mat}
For a torsion-free lattice $\G$ in $\PSL$ and a homomorphism $\rho : \G \to \Homeo$, we have
\begin{equation}\label{eq:MW}
    \lvert e(\rho) \rvert \leq -e(\G\backslash \D),
\end{equation}
where $e(\G\backslash \D)$ denotes the Euler characteristic of the hyperbolic surface $\G\backslash \D$.
Furthermore, we have $e(\rho) = e(\G\backslash \D)$ $($resp.\ $e(\rho) = -e(\G\backslash \D)$$)$ 
if and only if $\rho$ is semiconjugate to $\rho_0$ $($resp.\ $\overline{\rho}_{0}$$)$, where $\rho_{0}$ is the Fuchsian action of $\G$ induced from the inclusion $\G \to \PSL \to \Homeo$, and $\overline{\rho}_{0}$ is its complex conjugate. 
\end{cor}

In the case where $\G$ is uniform, $\G \backslash \D$ is a closed surface. In that case, Corollary \ref{cor:mat} is nothing but the Milnor--Wood inequality \cite{Milnor,Wood} and the rigidity theorem of Matsumoto \cite{Matsumoto1}. Note that Matsumoto \cite{Matsumoto3} recently gave an alternative concise proof of the theorem of Burger--Iozzi--Wienhard \cite{BIW} by using his rigidity theorem for closed surfaces. 

The key of the proof of the rigidity part of Corollary \ref{cor:mat} is to show a rigidity of harmonic measures in the equality case in the Milnor--Wood inequality \eqref{eq:MW}. It is well known that the pull back of a harmonic measure on the suspension bundle to $\D \times \R$ is of the form 
\begin{equation}\label{eq:disint0}
h(z,t) \vol(z) \nu (t) \quad (z \in \D, t \in \R),
\end{equation}
where $\vol$ is the leafwise volume form, $\nu$ is a Borel measure on $\R$ and $h$ is a positive leafwise harmonic function. Note that $h$ is determined only up to the multiplication of leafwise constant functions. We will show that, if we have the equality in \eqref{eq:MW}, then we can take $h$ closely related to the Poisson kernel as stated in the following result: 

\begin{thm}\label{thm:Poi}
Let $\G$ be a torsion-free lattice in $\PSL$ and $\rho : \G \to \Homeo$ a homomorphism. If we have $e(\rho) = e(\Sigma)$, then the pull back of every harmonic measure on the suspension bundle of $\rho$ to the universal cover $\D \times \R$ is of the form 
\[
\frac{1-\lvert z \rvert^2}{\lvert \m(e^{it})-z \rvert^2} \vol(z) \nu (t) \quad (z \in \D, t \in \R),
\]
where $\vol$ is the leafwise hyperbolic volume form, for a Borel measure $\nu$ on $\R$ and a continuous monotone mapping $\m : S^1 \to S^1$ of degree one which is $(\rho,\rho_0)$-equivariant, namely, $\m \circ \rho(\gamma) = \rho_0 (\gamma) \circ \m$ holds for $\forall \gamma \in \G$. 
\end{thm}

The map $\m$ gives a semiconjugacy from the  $\rho(\G)$-action to the Fuchsian $\G$-action. Matsumoto considered this map in \cite{Matsumoto2}. Consider a $\G$-action $\tau$ on a compact metric space $F$. Fix an ergodic harmonic measure on the suspension bundle. Then the lift to $\D \times F$ is of the form as in \eqref{eq:disint0}, where $\R$ is replaced with $F$. It is well known that, for each $t \in F$, the function $h(\cdot,t)$ is the integral of the Poisson kernel with respect to a Borel measure $\sigma_t$ on the ideal boundary of the leaf $\D \times \{t\}$. Matsumoto proved the following dichotomy: We have either the support of $\sigma_t$ is a point or the entire ideal boundary of the leaf for generic $t \in F$ simultaneously. In the former case, we have a $(\tau,\rho_0)$-equivariant measurable map $\m : F \to S^1; t \mapsto \operatorname{supp}(\sigma_t)$, which we call the \emph{Matsumoto map}. The map $\m$ in Theorem \ref{thm:Poi} is essentially this Matsumoto map.

\paragraph{Acknowledgments.}
We are grateful to Masayuki Asaoka for his comments, in particular, pointing out a serious gap in the proof of \eqref{eq:eachcusp} in a previous version of this article. We also thank Ken'ichi Yoshida for letting the authors know the publication of \cite{Thurston}. We thank the anonymous referee for careful reading and suggestions that highly improved the readability of this article.

M.A.\ is supported by JSPS KAKENHI Grant Numbers JP18K13422, JP19KK0347, JP21H00980, JP21K18579 and JP24K06776.
Y.M.\ is supported by JSPS KAKENHI Grant Number JP17K05260.
H.N.\ is supported by JSPS KAKENHI Grant Numbers JP20K03620 and JP24K06723.

\paragraph{Convention.}
Integrations and ``almost everywhere'' refer to the Lebesgue measure unless otherwise stated. 
The Dirac measure at a point $a$ is denoted by $\delta_a$.
We use the identification $S^1 = \R / 2\pi\Z$ throughout this paper.

\section{Preliminaries}\label{sec:prelim}

\subsection{Harmonic measures on the suspension foliation}

A harmonic measure on a foliated manifold is a measure on $M$ invariant under the leafwise heat flow, which was introduced by Garnett \cite{Garnett}. The advantage is that a nontrivial harmonic probability measure always exists for foliations on compact manifolds by Garnett's theorem, while transverse invariant measures may not exist. There are many applications to foliations, group actions and related topics, for instance, \cite{Yue,DK,Adachi2}.

Let $\G$ be a torsion-free lattice of $\PSL$, and $\Sigma :=\G\backslash \D$. For a homomorphism $\rho : \G \to \Homeo$, we can construct its suspension $M := \Sigma \times_\rho S^1 := \G\backslash (\D \times S^1) \to \Sigma$, where $\gamma \cdot (z,t) = (\gamma z,\rho(\gamma)t)$ for $\gamma \in \G, z\in \D, t\in S^1$. 
The suspension foliation $\F$ on $M$ is induced from the product foliation $\D \times S^1 = \sqcup_{w \in S^1}\D \times \{w\}$. 
Since $\G$ preserves the Poincar\'e metric 
\begin{equation}\label{eq:ghyp}
    (g_{\hyp})_{z} = \frac {4dz\,d{\bar{z}}}{(1-|z|^{2})^{2}} \quad (z \in \D),
\end{equation}
the leaves of $\F$ admit the natural Riemannian metric of constant curvature $-1$. 
We call it the leafwise hyperbolic metric.
Let us denote the leafwise Laplacian associated with the leafwise hyperbolic metric by $\Delta$. 

\begin{defn}
A Borel measure $\mu$ on $M$ is called \emph{harmonic} if 
\[
\int_M \Delta f(x) \mu(dx) = 0 
\]
for every compactly supported leafwise $C^2$ function $f$ on $M$ such that $\Delta f$ is continuous on $M$.
\end{defn}

If $M$ is compact, then a harmonic probability measure on $M$ exists by Garnett's theorem (see \cite{Garnett,Candel}). Now $M$ is not compact unless $\G$ is uniform. In this case, we need the following result of Alvarez \cite{Alvarez}:

\begin{thm}\label{thm:alvarez}
There exists a harmonic probability measure $\mu$ on $(M,\F)$.
\end{thm}

Alvarez \cite{Alvarez} constructed a harmonic probability measure by using a stationary measure for the $\G$-action on $S^1$. 
By the disintegration formula \cite{Garnett}, the pull back of a harmonic measure $\mu$ to $\D \times \R$ is of the form 
\begin{equation}\label{eq:disint}
\tilde{\mu} = h(z,t) \vol(z) \nu (t),
\end{equation}
where $h(z,t)$ is a Borel measurable function whose restriction to $\D \times \{t\}$ is a positive harmonic function for $\nu$-a.e.\ $t$, $\vol$ is the leafwise volume form and $\nu$ is a Borel measure on $\R$. 

\subsection{Harnack's inequality}
We equip the unit disk $\D = \{ z = x_1 + i x_2 \in \C \mid \lvert z \rvert < 1\}$ with the Poincar\'e metric and denote by $d_{\hyp}$ the induced hyperbolic distance, and by $\lvert \cdot \rvert_{\hyp}$ the induced norm on differential forms. The following is a fundamental inequality for positive harmonic functions on $\D$, which we apply to leafwise harmonic functions associated to harmonic measures.

\begin{lem}[Harnack's inequality] 
\label{lem:harnack}
Let $h$ be a positive harmonic function on $\D$.
\begin{enumerate}
\item For any $z, w \in \D$, $e^{-d_{\hyp}(z, w)}h(w) \leq h(z) \leq e^{d_{\hyp}(z, w)} h(w)$.
\item $\lvert d \log h \rvert_{\hyp} \leq 1$ on $\D$. 
\item If $\lvert d\log h \rvert_{\hyp} = 1$ at $\alpha \in \D$, then we have
\[
h(z) = h(0) \frac{1-\lvert z \rvert^2}{\lvert m-z \rvert^2},
\]
where
\[
c := \frac{h_{x_1}(\alpha) + ih_{x_2}(\alpha)}{\lvert h_{x_1}(\alpha) + ih_{x_2}(\alpha)\rvert} \in S^1, \quad m := \frac{c+ \alpha}{1 + \overline{\alpha}c}\in S^1.
\]
\end{enumerate}
\end{lem}
\begin{proof}
The first point is a standard fact (see, for instance, \cite{Ahlfors}). 
By taking the logarithm, we can deduce from the first point that 
\[
\lvert \log h(z) - \log h(w) \rvert \leq d_{\hyp}(z,w),
\]
hence, $\lvert d\log h \rvert_{\hyp} \leq 1$ holds everywhere on $\D$.

Let us show the third point. We use the Poisson formula for positive harmonic functions (see, for instance, \cite{Tsuji}), 
\begin{equation}\label{poisson}
h(z) = h(0) \int_{[0,2\pi)} \frac{1-\lvert z \rvert^2}{\lvert e^{it} - z \rvert^2} \mu(dt)
\end{equation}
where $\mu$ is a Borel measure on $\R$ with $\mu([0,2\pi)) = 1$ and invariant under $2\pi\Z$-translation.
Consider first the case when $\lvert d \log h \rvert_{\hyp} = 1$ holds at $\alpha = 0 \in \D$. By a rotation, we may assume $(\log h)_{x_1}(0) = 2$, $(\log h)_{x_2}(0) = 0$, $m = 1$. Differentiating \eqref{poisson} at $z = 0$, we have
\[
2 = (\log h)_{x_1}(0) = \int_{[0,2\pi)} \left.\left(\frac{\partial}{\partial x} \frac{1-\lvert z \rvert^2}{\lvert e^{it}-z \rvert^2}\right)\right\rvert_{z=0} {\mu(dt)} = \int_{[0,2\pi)} 2 \cos t\, {\mu(dt)}.
\]
Hence $1 \leq \cos t$ holds $\mu$-a.e.\ and 
$\operatorname{supp} \mu \subset 2\pi\Z$.
Therefore, $\mu = \sum_{n \in \Z} \delta_{2\pi n}$ and
\[
h(z) = {h(0)}\int_{[0,2\pi)} \frac{1-\lvert z \rvert^2}{\lvert e^{it}-z \rvert^2}\, \delta_0(dt) =  h(0) \frac{1-\lvert z \rvert^2}{\lvert 1-z \rvert^2}.
\]

Now let $\alpha \in \D$ arbitrary. 
Define $f_\alpha(z) = ({z + \alpha})/({1 + \overline{\alpha}z}) \in \operatorname{PSU}(1,1)$ and $g := h \circ f_\alpha$.
Then 
\[
g_{x_1}(0) = (1-\lvert \alpha \rvert^2)h_{x_1}(\alpha), \quad g_{x_2}(0) = (1- \lvert \alpha \rvert^2)h_{x_2}(\alpha),
\]
hence, $\lvert d \log g \rvert_{\hyp} = 1$ holds at $0 \in \D$. 
Applying the argument in previous paragraph to $g$, we obtain 
\[
g(z) = g(0)\frac{1-\lvert z \rvert^2}{\lvert c-z \rvert^2}, \quad 
c = \frac{h_{x_1}(\alpha) + i h_{x_2}(\alpha)}{\lvert h_{x_1}(\alpha) + i h_{x_2}(\alpha)\rvert},
\]
and it follows that
\[
h(z) = h(\alpha)\frac{1-\lvert f_\alpha^{-1}(z)\rvert^2}{\lvert c-f_\alpha^{-1}(z)\rvert^2} 
= 
h(\alpha)\frac{1-\lvert \alpha \rvert^2}{\lvert c+{\alpha}\rvert^2}. \frac{1-\lvert z \rvert^2}{\lvert \frac{c + \alpha}{1+\overline{\alpha}c} - z \rvert^2}.
\]
By letting $z = 0$, we see that 
$h(\alpha)\frac{1-\lvert\alpha\rvert^2}{\lvert c+{\alpha} \rvert^2} = h(0)$ and we complete the proof.
\end{proof}

\subsection{The Euler number for surface group actions on $S^1$}

Let $\G$ be a torsion-free lattice of $\PSL$, and $\rho : \G \to \Homeo$ a homomorphism. Let us recall the definition of the Euler number in \cite{BIW}. If $\G$ is not uniform, clearly we cannot adopt the classical definition by the pairing of the Euler class and the fundamental class of $\Sigma = \G\backslash \D$. 
Ghys \cite{Ghys} introduced the Euler class in the bounded cohomology. Burger--Iozzi--Wienhard \cite{BIW} defined the Euler number based on Ghys' idea by using the relative fundamental class in the bounded cohomology as follows:
Let $\Sigma'$ be the compact surface with boundary obtained by cutting off all cusps of $\Sigma$.
We can pull back the universal bounded Euler class $\eu \in H_{b}^{2}(\Homeo)$ by $\rho$ to have $\rho^*\eu \in H_{b}^{2}(\G)$, which we regard as an element of $H_{b}^{2}(\Sigma')$ via the so-called Gromov isomorphism $H^2_b(\G) \cong H^2_b(\Sigma')$. Since $H^1_b(\partial \Sigma')\cong H^2_b(\partial \Sigma')= 0$, by the relative exact sequence, we have an isomorphism
\[
f : H^2_b(\Sigma', \partial \Sigma') \to H^2_b(\Sigma').  
\]
The \emph{Euler number} $e(\rho) \in \R$ of $\rho$ is defined by
\[
e(\rho) = \langle f^{-1}\rho^*\eu, [\Sigma',\partial \Sigma'] \rangle,
\]
where $[\Sigma',\partial \Sigma']$ is the relative fundamental class of $(\Sigma',\partial \Sigma')$.

Like as the classical Euler number, by a result of \cite{BIW}, the Euler number $e(\rho)$ is expressed in terms of the suspension bundle $\pi \colon M = \Sigma' \times_{\rho} S^1 \to \Sigma'$ and the translation numbers of the action of the boundary loops of $\Sigma'$. Let us consider the case where $\partial \Sigma' \neq \emptyset$. Let $c_1, \ldots, c_m$ be the boundary loops of $\Sigma'$ whose orientations are induced from $\Sigma'$. Take a homomorphism lift $\tilde\rho : \Gamma \to 
\tHomeo$ of $\rho$, where 
\[
\tHomeo = \{ \, f \in \operatorname{Homeo}_+(\R) \mid f(x+2\pi) = f(x) + 2\pi, \forall x \in \R \, \}
\]
is the universal cover group of $\Homeo$. 
Since $\pi: M \to \Sigma'$ is trivial as an $S^1$-bundle, such lift exists. Let $\tau : \tHomeo \to \R$ be the translation number, which is defined by 
\[
\tau (f) = \frac{1}{2\pi} \lim_{n \to \infty} \frac{f^{n}(x) - x}{n},
\]
where $x$ is any point in $\R$ (see, e.g. \cite{Navas}). By \cite{BIW}, in the case where $\partial \Sigma' \neq \emptyset$, we have 
\[
e (\rho) = -\sum_{i=1}^{m} \tau(\tilde\rho (c_i)).
\]
Note that the lift $\tilde\rho$ determines a trivialization of the $S^1$-bundle $M\to \Sigma'$ up to homotopy, which is the image of the $0$-section of the suspension $\R$-bundle of $\tilde\rho$ under the projection $\Sigma' \times_{\tilde{\rho}} \R \to M$. 

Our Theorem \ref{thm:GB} states a Gauss--Bonnet formula that express this Euler number $e(\rho)$ in terms of a curvature integral with respect to a specifically chosen continuous connection.
To this end, let us recall the usual Gauss--Bonnet formula for surfaces with boundaries.
Let $\Sigma'$ be a smooth compact surface with smooth boundary, and $\pi: M \to \Sigma'$ a smooth principal $S^1$-bundle.
We will call a continuous $1$-form $\eta$ on $M$ a \emph{continuous $S^1$-connection form} on $M$ if $\eta$ is an $S^1$-invariant continuous section of $T^{*}M$ such that $\eta \left(\frac{\partial}{\partial \theta}\right)=1$ on $M$, where $e^{i\theta} \in S^1$ and $\frac{\partial}{\partial \theta}$ denotes the generator of the principal $S^1$-action.
For each piecewise smooth closed curve $\gamma$ in $\Sigma'$, $\hol_\eta(\gamma) \in S^1$ denotes the holonomy along $\gamma$ with respect to the parallel transport defined by the connection $\eta$.

\begin{defn}\label{def:conconn}
Let $\Omega$ be an integrable $2$-form on $\Sigma'$. 
For a (non-smooth) continuous $S^1$-connection form $\eta$ on $M$, we call  $\Omega$ the \emph{curvature} of $\eta$ if, for every simply-connected subset $U$ of $\Sigma'$ with piecewise smooth boundary $\partial U$, we have $\hol_\eta(\partial U) = \exp i \int_U \Omega$, where we regard as $\hol_\eta(\partial U) \in S^1$.
\end{defn}

When $\eta$ is smooth, it is well-known that its curvature exists as a smooth $2$-form $\Omega$ on $\Sigma'$ such that $\pi^*\Omega = d\eta$. 
In the step 4 of the proof of the main result (Theorem \ref{thm:GB}), we will use the following Gauss--Bonnet formula for continuous connections whose curvature is well-defined in the sense of Definition \ref{def:conconn}.

\begin{prop}\label{prop:GB}
Let $\Sigma'$ be a smooth compact surface with non-empty smooth boundary and $M \to \Sigma'$ a principal $S^1$-bundle with a global smooth section $\sigma$. Let $\eta$ be a continuous $S^1$-connection form on $M$. If $\eta$ has curvature $\Omega$ in the sense of Definition \ref{def:conconn}, then we have
\begin{equation}\label{eq:GB}
\frac{1}{2\pi} \int_{\Sigma'} \Omega = \sum_{i=1}^{m} \tau(\widetilde\hol_{\eta}(c_i)),  
\end{equation}
where $c_1, \dots, c_m$ are the boundary curves of $\Sigma'$ whose orientations are induced from  $\Sigma'$, 
and $\widetilde\hol_{\eta}(c_i) \colon \R \to \R$ is the lift of the holonomy map of $\eta$ along $c_i$ with respect to $\sigma$. 
\end{prop}

\begin{proof}
Trivialize $M \cong \Sigma' \times S^1$ so that the image of $\sigma$ is mapped to $\Sigma' \times \{1\}$. Consider the covering map $\varphi = \operatorname{id} \times p : \Sigma' \times \R \to \Sigma' \times S^1$, where $p \colon \R \to S^1$ is the standard map $p(t)= e^{it}$. Equip $\Sigma' \times \R$ with a connection $1$-form $\hat{\eta} = \varphi^*\eta$. 
Under this setting, the translation number of the holonomy of arbitrary paths with respect to $\sigma$ makes sense: For a path $c$ on $\Sigma'$, define  $\widehat{\hol}_{\eta}(c)= (\pr_2(\hat{c}(1)) - \pr_2(\hat{c}(0)))/2\pi$, where $\pr_2 \colon \Sigma' \times \R \to \R$ is the second projection and $\hat{c}$ is a horizontal lift of $c$ with respect to $\hat{\eta}$. 
Note that, since the holonomy along every path is a translation on $\R$, this translation number is independent of the choice of $\hat{c}$. Moreover, if $c$ is a piecewise smooth closed path, then we have $\tau(\widetilde\hol_{\eta}(c))=\widehat{\hol}_{\eta}(c)$.

Triangulate $\Sigma'$ piecewise smoothly. In order to prove \eqref{eq:GB}, it suffices to show 
\begin{equation}\label{eq:GB2}
\frac{1}{2\pi} \int_s \Omega = \tau(\widetilde\hol_{\eta}(\partial s))  
\end{equation}
 for each $2$-simplex $s$. Indeed we have $\tau(\widetilde\hol_{\eta}(\partial s)) = \sum_{k=0}^{2}\widehat\hol_{\eta}(t_k)$, where $t_0$, $t_1$, $t_2$ are the edges of $s$, and \eqref{eq:GB} is obtained as the sum of \eqref{eq:GB2} for all $2$-simplices. 
 Finally let us explain that \eqref{eq:GB2} follows from Definition \ref{def:conconn}. By definition, we have $\frac{1}{2\pi} \int_s \Omega - \tau(\widetilde\hol_{\eta}(\partial s)) \in \Z$. 
 Since $\Omega$ is integrable, 
 this integer depends continuously on $s$, and hence it is constant. When we contract $s$ to a point, the difference $\frac{1}{2\pi} \int_s \Omega - \tau(\widetilde\hol_{\eta}(\partial s))$  
 goes to zero, which implies \eqref{eq:GB2}. 
\end{proof}

\section{The Gauss--Bonnet formula for the connection associated with a harmonic measure}\label{sec:proof}

We shall prove a Gauss--Bonnet formula for the $S^1$-connection $\bar\omega$ associated with a harmonic measure (Theorem \ref{thm:GB}). Let $\G$ be a torsion-free lattice of $\PSL$, $\rho$ a homomorphism $\G \to \Homeo$ and $\Sigma = \G \backslash \D$. 
We assume that $\rho (\Gamma)$ has no finite orbit in $S^1$.
Consider the suspension foliation on $M=\Sigma \times_{\rho} S^1$, which is induced from $\{\D \times \{e^{it}\}\}_{e^{it} \in S^1}$, equipped with the standard leafwise hyperbolic metric.
The proof of Theorem \ref{thm:GB} is divided into the following four steps.

\subsection{Step 1: Collapsing the complement of the support of a harmonic measure}\label{subsect:semiconj}
In this first step, we construct a semiconjugacy of the given representation $\rho$ to a new representation $\rho'$ by integrating fiberwise measures to collapse the complement of the support of the given harmonic measure.

As mentioned above, the suspension foliation $\F$ has a leafwise hyperbolic metric. 
Take a harmonic probability measure $\mu$ on $\F$, 
which exists from Theorem \ref{thm:alvarez}. 
By multiplying a scalar, we normalize $\mu$ so that $\mu(M) = 4\pi^2 \lvert e(\Sigma)\rvert$.

Since the natural projection $\D \times \R \to \Sigma \times_{\rho} S^1$ is a covering map, we can pull back $\mu$ to $\D \times \R$ and denote it by $\tilde{\mu}$, which is a harmonic measure on the product foliation $\{\D \times \{t\}\}_{t \in \R}$.
By \eqref{eq:disint}, we may write
\begin{equation*}
\tilde{\mu} = q(z,t) \vol(z) \nu(t), 
\end{equation*}
where $\vol$ is the leafwise volume measure on $\D$, $\nu$ is a Borel measure on $\R$, $q$ is a locally integrable function on $\D \times \R$ with respect to $\vol(z) \nu(t)$ and, for $\nu$--a.e.\ $t$, $q(\cdot, t)$ is a positive harmonic function on $\D$. 

For each $z \in \D$, consider a measure $\mu_z$ on $S^1$ induced by $\tilde{\mu}_z := q(z,t)\nu(t)$ on $\R$. Since $\tilde{\mu}$ is $\G$-invariant by construction, it follows that, for each $\gamma \in \Gamma$, we have
\begin{equation}\label{eq:fibermeas}
\mu_{\gamma z} = \rho(\gamma)_* \mu_z.
\end{equation}

\begin{claim}\label{nonatomic}
$\mu_z$ is a non-atomic measure for every $z \in \D$.
\end{claim}

\begin{proof} 
Assume on the contrary that $\mu_{z_0}(\{a\}) > 0$ for some $z_0 \in \D$ and some $a \in S^1$.
Then, $\nu(\{a\}) > 0$.
Let $L$ be the leaf of $\mathcal{F}$ induced from $\D \times \{a\}$. 
The measure $\nu(\{a\}) q(z,a) \vol(z) \delta_a(t)$ supported on $\D \times \{a\}$ induces a measure $\mu_L$ supported on $L$. Then,
\[
0 < \mu_L(L) \leq \mu(M) < \infty
\]
and $q(z,a)$ induces an $L^1(\vol)$ positive harmonic function on $L$.
Since $\rho(\Gamma)$ is assumed to have no finite orbit in $S^1$, any leaf of $\mathcal{F}$, being an infinite covering of $\Sigma$, is a complete hyperbolic surface of infinite volume. 
This leads to a contradiction by \cite[Theorem 2.4]{Li-Schoen} or \cite[Theorem 1]{Li}. 
\end{proof}

\begin{claim}
$\mu_z(S^1) = 2\pi$ for every $z \in \D$.
\end{claim}
\begin{proof}
From the invariance \eqref{eq:fibermeas}, $\mu_z(S^1)$ defines an $L^1(\vol)$ positive harmonic function on $\Sigma$.
This function must be constant from \cite[Theorem 2.4]{Li-Schoen} or \cite[Theorem 1]{Li}. Since $\mu$ is normalized so that $\mu(M) = 4\pi^2|e(\Sigma)|$ and the volume of $\Sigma$ is $2\pi|e(\Sigma)|$, this constant is $2\pi$.
\end{proof}

Define a monotone map $\tilde{\psi}: \R \to \R$ by 
\[
\tilde{\psi}(t) := \int_0^t \tilde\mu_0(ds)  = \int_0^t q(0,s) \nu(ds).
\]
By construction, $\tilde{\psi}$ induces a monotone continuous map  $\psi: S^1 \to S^1$ of mapping degree one. 

\begin{claim}
\label{newholonomy}
There exists a group homomorphism $\rho': \Gamma \to \mathrm{Homeo}_+(S^1)$ such that $\psi$ is $(\rho,\rho')$-equivariant, namely, we have
\begin{equation}\label{semiconj}
\psi \circ \rho(\gamma) = \rho'(\gamma) \circ \psi \quad\quad (\forall \gamma\in\Gamma).
\end{equation}
\end{claim}

\begin{proof}
Fix $\gamma \in \Gamma$. In order to define $\rho'(\gamma): S^1 \to S^1$, take $e^{it'} \in S^1$. Since $\psi$ is surjective, there exists $e^{it} \in S^1$ such that $\psi(e^{it}) = e^{it'}$. Then we define $\rho'(\gamma) e^{it'} := \psi(\rho(\gamma)e^{it})$. It is easy to see that this $\rho'$ is well-defined, and that $\rho'$ is a group homomorphism $\Gamma \to \Homeo$. The equation \eqref{semiconj} follows from the construction. 
\end{proof}

Actually, $\rho'$ is a representation of $\Gamma$ in the bi-Lipschitz homeomorphism group of $S^1$, as it will follow from the arguments in next step.

\subsection{Step 2: Construction of a smooth structure}\label{subsect:smooth-str}
Consider the suspension bundle $\pi' : \Sigma \times_{\rho'} S^1 \to \Sigma$ for $\rho'$
and let $\F'$ denote the suspension foliation on $M' := \Sigma \times_{\rho'} S^1$.
In this step, we construct a smooth structure on our $S^1$-bundle $M'$ with respect to which $\F'$ is a transversely Lipschitz foliation.

Using $\psi$ constructed in the first step, we define 
\[
\tilde{\Psi}: \D \times \R \to \D \times \R, \quad \tilde{\Psi}(z,t) := (z, \tilde\psi(t))
\]
that induces a surjective continuous map $\Psi: M \to M'$. 
Let us consider the push-forward measures $\mu' := \Psi_* {\mu}$ and $\tilde{\mu}' := \tilde{\Psi}_* {\tilde{\mu}}$,
which are harmonic measures on $\F'$ and the product foliation $\{ \D \times \{t\} \}_{t \in \R}$, respectively. 
By \eqref{eq:disint}, like as $\tilde{\mu}$, we have
\begin{equation}\label{eq:harmstr}
\tilde{\mu}' = h(z,t) \vol(z) \lambda(t), 
\end{equation}
where $\lambda$ is the Lebesgue measure on $\R$ and 
$h$ is a locally integrable function on $\D \times \R$ such that, for a.e.\ $t$, $h(\cdot, t)$ is a positive harmonic function on $\D$. 
In particular, $h(0,t)=1$ holds for a.e.\ $t$.
Note that $\tilde\mu'_z := h(z,t)\lambda(t)$ induces a non-atomic measure $\mu'_z$ on $S^1$ with $\mu'_z(S^1) = 2\pi$ for each $z \in \D$.

To construct a desired smooth structure on $M'$, we integrate the harmonic measure $\mu'$ on each fiber of $\pi': M' \to \Sigma$. 
The next claim, a direct consequence of Harnack's inequality, is a key ingredient of our proof and will justify analytic arguments in next steps.

\begin{lem}\label{lem:bilipschitz}
Define a map $\tilde{\Phi}: \D \times \R \to \D \times \R$ by $\tilde{\Phi}(z,t) := (z, \varphi(z,t))$, where
\[
\varphi(z,t):=\int_0^t h(z,s) ds
\]
and $h$ is the function that appeared in \eqref{eq:harmstr}.
Then $\tilde{\Phi}$ is a locally bi-Lipschitz homeomorphism. In particular, for every $z \in \D$, the map $\varphi(z,\cdot)$ is a locally bi-Lipschitz homeomorphism of $\R$, whose inverse is denoted by $\tau(z,\cdot)$.
\end{lem}

\begin{proof}
It is clear that $\tilde{\Phi}$ is bijective. To show that $\tilde{\Phi}$ is locally Lipschitz, it is enough to see that $\varphi: \D \times \R \to \R$ is Lipschitz in each variable on 
\[
D_N := \{ (z,t) \in \D \times \R \mid d_{\hyp}(0,z) \leq N, \lvert t \rvert \leq 2\pi N\}
\]
for arbitrary $N \in \N$.
Since Lemma \ref{lem:harnack} implies that, for a.e.\ $t$,
\[
0 < h(z, t) \leq e^{d_{\hyp}(0,z)} h(0,t) = e^{d_{\hyp}(0,z)} \leq e^N,
\]
 we have, for $(z,t_1), (z,t_2) \in D_N$,
\[
 \lvert \varphi(z, t_1) - \varphi(z, t_2) \rvert
 = \left\lvert \int_{t_2}^{t_1}h(z,t) dt \right\rvert
 \leq e^{N}\lvert t_1 - t_2\rvert.
\]
Next, let $(z_1,t), (z_2,t) \in D_N$. Lemma \ref{lem:harnack} also yields that, for a.e.\ $t$, 
\[
\lvert h(z_1, t) - h(z_2, t)\rvert
\leq (e^{d_{\hyp}(z_1,z_2)}-1)\max \{ h(z_1, t), h(z_2, t)\}.
\]
Since the exponential function is convex, we have
\[
e^{d_{\hyp}(z_1,z_2)}-1 \leq \frac{e^{2N}}{2N} d_{\hyp}(z_1,z_2).
\]
Then, we have
\begin{equation}\label{eq:lipschitz1}
\lvert {\varphi}(z_1, t) - {\varphi}(z_2, t)\rvert
\leq \left\lvert \int_{0}^{t} \lvert h(z_1,s) - h(z_2, s)\rvert ds  \right\rvert 
\leq \pi e^{3N} d_{\hyp}(z_1, z_2).
\end{equation}
Note that this inequality holds also for $t < 0$. Therefore, $\varphi$ is locally Lipschitz, and hence so is $\tilde{\Phi}$.

It remains to show that $\tilde{\Phi}^{-1}$ is locally Lipschitz. 
We denote the inverse map of $\varphi(z, \cdot): \R \to \R$ by $\tau(z, \cdot)$ for each $z \in \D$. Again, it is enough to see that $\tau$ is Lipschitz in each variable on $D_N$ for arbitrary $N \in \N$.
Let $(z,\theta_1),(z,\theta_2) \in D_N$ and denote $t_j := \tau(z, \theta_j)$ for $j = 1, 2$.
From Lemma \ref{lem:harnack}, we have
\[
h(z, t) \geq e^{-d_{\hyp}(0,z)} h(0,t) = e^{-d_{\hyp}(0,z)} \geq e^{-N}
\]
for a.e.\ $t$. Hence, it follows that
\[
\lvert \varphi(z,t_1) - \varphi(z,t_2) \rvert
 = \left\lvert \int_{t_2}^{t_1}h(z_1,t) dt \right\rvert
 \geq e^{-N} \lvert t_1 - t_2 \rvert,
\]
which is equivalent to 
\begin{equation} \label{eq:lipschitz2}
 \lvert \tau(z,\theta_1) - \tau(z,\theta_2) \rvert 
 \leq e^{N} \lvert \theta_1 - \theta_2 \rvert.
\end{equation}
Next, for each $(z_1,\theta),(z_2,\theta) \in D_N$, we deduce from \eqref{eq:lipschitz2} and \eqref{eq:lipschitz1} that
\begin{align*}
\lvert\tau(z_1,\theta) - \tau(z_2,\theta)\rvert
&= \lvert \tau(z_2, \varphi(z_2, \tau(z_1, \theta))) - \tau(z_2,\theta)\rvert\\
& \leq e^N \lvert \varphi(z_2, \tau(z_1,\theta)) - \theta \rvert\\
& = e^N \lvert \varphi(z_2, \tau(z_1,\theta)) - \varphi(z_1, \tau(z_1,\theta))\rvert\\
& \leq \pi e^{4N} d_{\hyp}(z_1, z_2).
\end{align*}
This completes the proof.
\end{proof}

Using $\tilde{\Phi}$, we define a new $\Gamma$-action on $\D \times S^1$ by
\begin{equation} \label{smoothaction}
\gamma \cdot (z, e^{i\theta}) := 
(\gamma z, e^{i{\varphi}(\gamma z,  {\tilde{\rho}'(\gamma)}(\tau(z,\theta)))})
\end{equation}
where $\gamma \in \Gamma$, $(z,e^{i\theta}) \in \D \times S^1$ and $\tilde{\rho}'(\gamma)$ denotes an arbitrary lift of $\rho'(\gamma) \in \mathrm{Homeo}_+(S^1)$ to $\mathrm{Homeo}_+(\R)$. 
We see from the following claim that this $\Gamma$-action is smooth.   
\begin{claim}
$\varphi(\gamma z, \tilde{\rho}'(\gamma)( \tau(z,\theta)))$ is a smooth function on $\D \times \R$.
\end{claim}
\begin{proof}
Using the invariance $\mu'_{\gamma z} = \rho'(\gamma)_* \mu'_z$, we have
\begin{align*}
\varphi(\gamma z, \tilde{\rho}'(\gamma)( \tau(z,\theta)))
& = \int_0^{\tilde{\rho}'(\gamma)(\tau(z,\theta))} \mu'_{\gamma z} (dt)\\
& = \int_{\tilde{\rho}'(\gamma)( 0)}^{\tilde{\rho}'(\gamma)(\tau(z,\theta))} \mu'_{\gamma z} (dt)
+ \int_0^{\tilde{\rho}'(\gamma)(0)} \mu'_{\gamma z} (dt) \\
& = \int_{0}^{\tau(z,\theta)} \mu'_{z} (dt)
+ \int_0^{\tilde{\rho}'(\gamma)(0)} \mu'_{\gamma z} (dt) \\
& = \varphi(z,\tau(z,\theta)) + \int_0^{\tilde{\rho}'(\gamma)(0)} \mu'_{\gamma z} (dt) \\
& = \theta + \int_0^{\tilde{\rho}'(\gamma)(0)} h({\gamma z}, t) dt.
\end{align*}
The last term $ \int_0^{\tilde{\rho}'(\gamma)(0)} h({\gamma z},t) dt$ is a positive harmonic function in $z$, hence, smooth in $z$.
\end{proof}

Using the $\Gamma$-action given in \eqref{smoothaction}, we define a smooth $S^1$-bundle $P := \Gamma \setminus (\D \times S^1)$.
The bi-Lipschitz homeomorphism $\widetilde{\Phi}$ induces a homeomorphism $\Phi \colon M' \to P$.

\subsection{Step 3: The $S^1$-connection associated with a harmonic measure}\label{subsect:curvature-est}
In this step, we construct a continuous connection form on $P$ whose curvature is bounded by one everywhere in its modulus. 

Let $\mathcal{F}_P$ denote the transversely Lipschitz foliation on $P$ obtained by mapping the suspension foliation $\mathcal{F}'$ on $M'$ by $\Phi$. Let $\mathcal{F}_{\widetilde{P}}$ denote its lift to $\widetilde{P} := \D \times \R$. 
Note that the leaves of $\mathcal{F}_{\widetilde{P}}$ are the graphs of positive harmonic functions $\varphi(\cdot, t): \D \to \R$.
Let $z = x_1 + ix_2 \in \D$.
Let us compute the slope of this graph. Recall that, for each $z \in \D$, the map $\tau(z,\cdot) \colon \R \to \R$ is the inverse of a locally bi-Lipschitz homeomorphism $\varphi(z,\cdot)$, which was defined by $\varphi(z,t) = \int_0^t h(z,s) ds$ in Lemma \ref{lem:bilipschitz}. For $j = 1, 2$, let 
\[
\omega_j(z, \theta) := \left.\frac{\partial \varphi}{\partial x_j}(z, t)\right\rvert_{t = \tau(z,\theta)}.
\]

\begin{lem} \label{lem:bddness}
The function $\omega_j$ is continuous on $\widetilde{P}$.
For each $z \in \D$, $\omega_j(z, \cdot)$ is a periodic Lipschitz function on $\R$ of period $2\pi$.
\end{lem}

\begin{proof}
By Lemma \ref{lem:harnack}, we have 
\[
\left\lvert \frac{\partial h}{\partial x_j} (z, t)\right\rvert \leq \frac{2}{1-\lvert z \rvert^2} h(z, t) \leq 
\frac{2}{1-\lvert z \rvert^2} e^{d_{\hyp}(0,z)},
\]
which implies that, for any compact subset $K \subset \D$, it is bounded from above on $K \times \R$.
Therefore we can change the order of differential and integration to have
\[
\frac{\partial \varphi}{\partial x_j}(z, t) = \frac{\partial}{\partial x_j} \int_0^{t} h(z, s) ds
=\int_0^{t} \frac{\partial h}{\partial x_j} (z, s) ds.
\]
From this expression and Lemma \ref{lem:bilipschitz} it follows that  $\omega_j$ is continuous.

We also have
\begin{align*}
\left\lvert \frac{\partial \varphi}{\partial x_j}(z, t_1) - \frac{\partial \varphi}{\partial x_j}(z, t_2) \right\rvert 
& \leq \int_{t_2}^{t_1} \left\lvert \frac{\partial h}{\partial x_j}(z,s)\right\rvert ds\\
& \leq \frac{2}{1-\lvert z \rvert^2} \int_{t_2}^{t_1} h(z, s) ds\\
&= \frac{2}{1-\lvert z \rvert^2} \lvert\varphi(z, t_1) - \varphi(z, t_2)\rvert
\end{align*}
for any $t_1 \geq t_2$.
Therefore, we conclude that $\omega_j(z, \cdot)$ is Lipschitz function on $\R$:
\[
\left\lvert \omega_j(z, \theta_1) - \omega_j(z, \theta_2) \right\rvert 
\leq \frac{2}{1-\lvert z \rvert^2} \lvert \theta_1 - \theta_2 \rvert.
\]
The periodicity of $\omega_j(z, \cdot)$ is clear. 
\end{proof}

\begin{rem}
We do not know if $\omega_j(\cdot, \theta)$ is locally Lipschitz or not, while it is continuous.
\end{rem}

We will construct a connection on $P$ by taking average of the slope $\omega_j(z, \cdot)$; consider an $S^1$-invariant continuous $1$-form 
\[
\overline{\omega} := d\theta - \sum_{j=1,2}\left( \int_0^{2\pi} \omega_j(z, \theta) \frac{d\theta}{2\pi} \right) dx_j
\]
on $\D \times S^1$.
Since $\overline{\omega}$ is $\Gamma$-invariant, we can identify it with a continuous $1$-form on $P$. This gives a continuous connection on $P$.

Let us show that $\overline{\omega}$ has a curvature form. Define
\begin{align*}
K(z) &:= -\frac{(1-\lvert z \rvert^2)^2}{4}\int_0^{2\pi} \left(-\frac{\partial \omega_1}{\partial \theta}(z, \theta) \omega_2(z, \theta) + \frac{\partial \omega_2}{\partial \theta}(z, \theta) \omega_1(z, \theta) \right) \frac{d\theta}{2\pi},
\end{align*}
and regard the $\G$-invariant form $K(z) \vol(dz)$ as a measurable $2$-form on $\Sigma$. Recall that the hyperbolic volume form is of the form $\frac{4dz \wedge d{\bar{z}}}{(1-|z|^{2})^{2}} \, (z \in \D)$.

\begin{lem} \label{lem:curvature}
$\overline{\omega}$ has the curvature $K(z) \vol(dz)$ in the sense of Definition \ref{def:conconn}.
\end{lem}

\begin{proof}
Take a piecewise smooth simple closed curve $\gamma: [0,1] \to \D$. Let $U$ be the domain bounded by $\gamma$. 
It is enough to show
\[
\hol_{\overline{\omega}}(\gamma) = \exp i\int_U K(z) \vol(dz)
\]
assuming that $\gamma$ is smooth.
Consider the horizontal lift $\tilde{\gamma}: [0,1] \to \D \times \R$ of $\gamma$ to $\widetilde{P}$ with respect to $\bar{\omega}$ whose initial value is $\tilde{\gamma}(0)=0$.
For $\exp i\tilde{\gamma}(1) = \hol_{\overline{\omega}}(\gamma)$, we need to show $\tilde{\gamma}(1) =  \int_U K(z) \vol(dz)$.

Since $\tilde{\gamma}$ is a solution of an ordinary differential equation
\[
\frac{d \tilde{\gamma}}{du}(u) = \sum_{j=1,2} \left(\int_0^{2\pi} \omega_j(\gamma(u), \theta)\frac{d\theta}{2\pi}\right) \frac{d \gamma_j}{du}(u)
\]
and the right hand side does not contain $\tilde{\gamma}$, we can integrate it to compute the solution:
\begin{align}
\tilde{\gamma}(1) 
& =  \int_0^{1}  du  \sum_{j=1,2} \left(\int_0^{2\pi} \omega_j(\gamma(u), \theta)\frac{d\theta}{2\pi}\right) \frac{d \gamma_j}{du}(u) \notag \\
& = \int_0^{2\pi} \frac{d\theta}{2\pi}   \int_\gamma \omega_1(z, \theta) dx_1 +  \omega_2(z, \theta) dx_2.
\label{eq:holonomy}\end{align}
We would like to apply the Stokes theorem to compute the path integral, but the difficulty here is that $\omega_j(z,\theta)$'s are merely continuous in $z$.
We need to approximate $\varphi(z,t)$ by a sequence of smooth functions.
By using mollifiers, we can construct smooth functions $\{ f_n(z,t)\}$  on $\overline{U} \times [0,2\pi]$ such that 
\begin{align*}
&\sup_{\overline{U} \times [0,2\pi]} \lvert f_n - \varphi \rvert 
+ \sum_{j=1,2} \sup_{\overline{U} \times [0,2\pi]} \left\lvert \frac{\partial f_n}{\partial x_j} - \frac{\partial \varphi}{\partial x_j}\right\rvert \to 0,\\
&\sum_{j=1,2} \int_{\overline{U} \times [0,2\pi]} \left\lvert \frac{\partial^2 f_n}{\partial t \partial x_j} - \frac{\partial^2 \varphi}{\partial t \partial x_j}\right\rvert dx_1 dx_2 dt\to 0
\end{align*}
as $n \to \infty$. Note that $\varphi(z,t)$ and $\frac{\partial \varphi}{\partial x_j}(z,t)$ are continuous, and $\frac{\partial \varphi}{\partial t}(z,t) = h(z,t)$ and $\frac{\partial^2 \varphi}{\partial t \partial x_j}(z,t) = \frac{\partial h}{\partial x_j}(z,t)$ are bounded measurable and integrable from Lemma \ref{lem:harnack} and \ref{lem:bddness}.
Moreover, from Lemma \ref{lem:bilipschitz},
\[
\int_{\overline{U} \times [0,2\pi]} \left\lvert \frac{\partial^2 f_n}{\partial t \partial x_j}(z,\tau(z,\theta)) - \frac{\partial^2 \varphi}{\partial t \partial x_j}(z,\tau(z,\theta)) \right\rvert dx_1 dx_2 d\theta\to 0
\]
follows. Hence, Fubini's theorem  yields that, for a.e.\ $\theta$,
\begin{equation} \label{approximation}
\int_{\overline{U}} \left\lvert \frac{\partial^2 f_n}{\partial t \partial x_j}(z,\tau(z,\theta)) - \frac{\partial^2 \varphi}{\partial t \partial x_j}(z,\tau(z,\theta)) \right\rvert dx_1 dx_2 \to 0.
\end{equation}

Using this sequence, we can approximate 
\[
\int_\gamma \omega_1(z, \theta) dx_1 +  \omega_2(z, \theta) dx_2
 = \lim_{n\to \infty} \int_\gamma \frac{\partial f_n}{\partial x_1}(z, \tau(z,\theta)) dx_1 +  \frac{\partial f_n}{\partial x_2}(z, \tau(z,\theta)) dx_2.
\]
From Lemma \ref{lem:bilipschitz}, the integrand in the approximation sequence is a Lipschitz $1$-form and we can apply the Stokes theorem for it. 
Note that we can show the Stokes theorem for Lipschitz $1$-forms by using that fact that the second fundamental theorem of calculus holds for Lipschitz functions (see, e.g., \cite[Section 1.6.4]{Tao}). Note also that the Stokes theorem holds in more general situations (see \cite[Theorem 4.5.6]{Federer}).

It follows that 
 \begin{align*}
 &\int_\gamma \frac{\partial f_n}{\partial x_1}(z, \tau(z,\theta)) dx_1 +  \frac{\partial f_n}{\partial x_2}(z, \tau(z,\theta)) dx_2\\
= & \int_U \left( -\frac{\partial}{\partial x_2} \frac{\partial f_n}{\partial x_1}(z, \tau(z,\theta)) + \frac{\partial}{\partial x_1} \frac{\partial f_n}{\partial x_2}(z, \tau(z,\theta)) \right)dx_1 dx_2\\
 = &  \int_U \left( -\frac{\partial^2 f_n}{\partial t \partial x_1}(z, \tau(z,\theta)) \frac{\partial \tau}{\partial x_2}(z, \theta) + \frac{\partial^2 f_n}{\partial t \partial x_2}(z, \tau(z,\theta)) \frac{\partial \tau}{\partial x_1}(z,\theta) \right)dx_1 dx_2.
\end{align*}
Notice that $\frac{\partial^2 f_n}{\partial x_1 \partial x_2}$ is canceled out. 
By letting $n \to \infty$, \eqref{approximation} implies
\begin{align*}
&\int_\gamma \omega_1(z, \theta) dx_1 +  \omega_2(z, \theta) dx_2 \\
&= \int_U \left( -\frac{\partial^2 \varphi}{\partial t \partial x_1}(z, \tau(z,\theta)) \frac{\partial \tau}{\partial x_2}(z, \theta) + \frac{\partial^2 \varphi}{\partial t \partial x_2}(z, \tau(z,\theta)) \frac{\partial \tau}{\partial x_1}(z,\theta) \right)dx_1 dx_2
\end{align*}
for a.e.\ $\theta$. On the other hand, thanks to Lemma \ref{lem:bilipschitz}, we can differentiate $\varphi(z, \tau(z,\theta)) = \theta$ by the chain rule and see that
\[
\frac{\partial \varphi}{\partial x_j}(z, \tau(z, \theta)) + \frac{\partial \varphi}{\partial t}(z, \tau(z, \theta)) \frac{\partial \tau}{\partial x_j}(z, \theta) = 0
\]
holds for a.e.\ $z$ and all $\theta$. Also,
\[
\frac{\partial \varphi}{\partial t}(z, \tau(z, \theta)) \frac{\partial \tau}{\partial \theta}(z, \theta) = 1
\]
is true for a.e.\ $\theta$ and all $z$. Hence, it follows that, for a.e.\ $\theta$ and a.e.\ $z$, 
\[
 \frac{\partial \tau}{\partial x_j}(z, \theta)
  = -\omega_j(z, \theta) \frac{\partial \tau}{\partial \theta}(z, \theta)
\]
holds, and we have
\begin{align*}
 &\int_\gamma \omega_1(z, \theta) dx_1 +  \omega_2(z, \theta) dx_2\\
& =  \int_U \left(\frac{\partial \omega_1}{\partial \theta}(z, \theta) \omega_2(z, \theta) - \frac{\partial \omega_2}{\partial \theta}(z, \theta) \omega_1(z, \theta) \right)dx_1 dx_2
\end{align*}
for a.e.\ $\theta$. Therefore, by changing the order of integrals in \eqref{eq:holonomy}, we have $\tilde{\gamma}(1) =\int_U  K(z) \vol(dz)$.
\end{proof}

Applying the isoperimetric inequality and Harnack's inequality, we show that $\lvert K(z) \rvert$ is bounded by one everywhere.

\begin{claim}
\label{inequality}
For every $z \in \D$, we have $\lvert K(z) \rvert \leq 1$. 
\end{claim}
\begin{proof}
Among the definition of $K(z)$,
\[
\int_0^{2\pi} \frac{1}{2}\left(-\frac{\partial \omega_1}{\partial \theta}(z, \theta) \omega_2(z, \theta) + \frac{\partial \omega_2}{\partial \theta}(z, \theta) \omega_1(z, \theta) \right) d\theta
\]
is the signed area of the domain bounded by a Lipschitz closed curve $(\omega_1(z,\theta), \omega_2(z,\theta))$ $(0 \leq \theta \leq 2\pi)$ on $\R^2$. The tangent vector of this curve is defined for a.e.\ $\theta$ and it is 
\begin{align*}
\frac{\partial}{\partial \theta}\omega_j(z, \theta) &= \frac{\partial}{\partial\theta}\int_0^{\tau(z,\theta)} \frac{\partial h}{\partial x_j}(z,t) dt = \frac{\partial \tau}{\partial \theta}(z, \theta) \frac{\partial h}{\partial x_j}(z, \tau(z,\theta))\\
&= \left(\frac{\partial \varphi}{\partial t}(z, \tau(z,\theta))\right)^{-1} \frac{\partial h}{\partial x_j}(z,\tau(z,\theta)) = \frac{\partial \log h}{\partial x_j}(z, \tau(z,\theta)).
\end{align*}
Then, by the isoperimetric inequality, we have
\begin{multline*}
\lvert K \rvert \leq \frac{(1-\lvert z \rvert^2)^2}{4\pi} \cdot \frac{1}{4\pi} \left(\int_0^{2\pi} \sqrt{\left(\frac{\partial \log h}{\partial x_1}\right)^2 + \left(\frac{\partial \log h}{\partial x_2}\right)^2} d\theta\right)^2 \\
= \frac{1}{4\pi^2} \left(\int_0^{2\pi} \lvert d \log h \rvert_{\hyp} d\theta\right)^2,
\end{multline*}
which implies $\lvert K \rvert \leq (2\pi)^2/{4\pi^2} = 1$ by Harnack's inequality. Note that the factor $\frac{(1-\lvert z \rvert^2)^2}{4}$ comes from the Poincar\'e metric on $\D$ (see \eqref{eq:ghyp}).
\end{proof}
This completes the proof for the first part of Theorem \ref{thm:GB}.

\subsection{Step 4: The translation number of the holonomy along horocircles}\label{subs:tr}
We will show the remaining part of Theorem \ref{thm:GB}, namely, the Gauss--Bonnet formula
$e(\rho) = \frac{1}{2\pi} \int_{\Sigma} K(z) \vol (dz)$.
Let $\rho'$ be the $\G$-action on $S^1$ obtained in Step 1 (Section \ref{subsect:semiconj}) by collapsing the complement of the harmonic measure $\mu$. The suspension foliation of $\rho'$ has the harmonic measure induced from $\mu$, which was denoted by $\mu'$. Let us consider the $S^1$-connection $\bar\omega$ on the suspension bundle of $\rho'$ constructed in Step 3 (Section \ref{subsect:curvature-est}), which is the fiberwise $\mu'$-average of the flat connection. Since the Euler class in the bounded cohomology is invariant under semiconjugacy due to an observation of Ghys \cite{Ghys} (see also \cite[Theorem 4.3]{BFH}), we have $e(\rho) = e(\rho')$. Therefore, in order to prove Theorem \ref{thm:GB}, it suffices to show
\begin{equation}\label{eq:GB34}
e(\rho') = \frac{1}{2\pi} \int_{\Sigma} K(z) \vol (dz).
\end{equation}

This is well-known when $\Sigma$ is a closed surface. We discuss the case where $\Sigma$ has cusps. 
Take a homomorphism lift $\tilde\rho' : \Gamma \to \tHomeo$ of $\rho'$. Then, by \cite{BIW}, we have
\begin{equation}\label{eq:tr1}
e(\rho') = - \sum_{i=1}^{m}\tau (\tilde\rho'(c_i)),
\end{equation}
where $\tau : \tHomeo \to \R$ is the translation number and $c_i \in \Gamma$ corresponds to a loop that goes around the $i$-th cusp of $\Sigma$ for $i=1, \dots, m$. 
A neighborhood of each cusp is foliated by closed horocircles. For each $i$, let $\{c_i^s\}_{s \gg 0}$ be the family of horocircles in a neighborhood of the $i$-th cusp so that
the hyperbolic diameter $\delta^s_i$ of $c_i^s$ tends to 0   
and $c_i^s$ approaches to the $i$-th cusp as $s \to \infty$. 
For $s \gg 0$, let $\Sigma^s$ be the compact subsurface of $\Sigma$ bounded by $c_1^s, \dots, c_m^s$. 
We denote by $\sigma$ the global section of $P \to \Sigma$ that corresponds to the choice of the lift $\tilde\rho'$ of ${\rho'}$.
For $s \gg 0$, by Proposition \ref{prop:GB}, we have
\begin{equation}\label{eq:tr2}
 \frac{1}{2\pi} \int_{\Sigma^s}K(z) \vol (dz)= \sum_{i=1}^{m} \tau(\widetilde{\hol}_{\bar\omega}(c_i^s)),
\end{equation}
where $\widetilde{\hol}_{\bar\omega}(c_i^s) : \R \to \R$ is the lift of the holonomy map of $\bar\omega$ along $c_i^s$ with respect to $\sigma$. Since $\Gamma \backslash \D$ has finite volume and $\lvert K(z) \rvert \leq 1$ a.e.\ $z$ from Claim \ref{inequality}, we have  $\int_{\Sigma^s} K(z) \vol(dz) \to \int_\Sigma K(z) \vol(dz)$ as $s \to \infty$. 
Therefore, by \eqref{eq:tr1} and \eqref{eq:tr2}, for the proof of the Gauss--Bonnet formula \eqref{eq:GB34}, it remains to find a sequence $\{s_n\}$ such that 
\begin{equation}\label{eq:eachcusp}
\tau(\widetilde\hol_{\bar\omega}(c_i^{s_n})) \to -\tau (\tilde\rho'(c_i))
\end{equation}
as $n \to \infty$ for each fixed $i$. 

Take a lift of $c_i^s$ on $\D$, $\gamma^s = \gamma^s_1 + i\gamma^s_2: [0,1] \to \D$, so that it is continuous in $s$. 
We use the coordinate $(z, t)$ and $(z, \theta)$ of $\D \times \R$ used in Section \ref{subsect:curvature-est} for the pull back bundle of $\Sigma \times_{\tilde{\rho}'} \R$ over $\D$,
and write $\widetilde{P}$ for $\D \times \R$ equipped with $(z,\theta)$-coordinate. Denote by $\tilde{\sigma}$ a lift of $\sigma$ to a global section of $\D \times \R \to \D$, and express $\tilde{\sigma}$ as the graph of a smooth function $\tilde{\sigma}: \D \to \R$ in $(z, \theta)$-coordinate.

Since $(\gamma^s(0), \widetilde{\sigma}(\gamma^s(0))), (\gamma^s(1), \widetilde{\sigma}(\gamma^s(1))) \in \widetilde{P}$ are identified in $\Sigma \times_{\tilde{\rho}'} \R$, it follows from \eqref{eq:holonomy} that
\[
2\pi \tau(\widetilde\hol_{\overline\omega}(c_i^s)) = \int_0^{2\pi} \frac{d\theta}{2\pi} \int_{\gamma^s} (\omega_1(z, \theta) dx_1 + \omega_2(z, \theta) dx_2) + \tilde{\sigma}(\gamma^s(0)) - \tilde{\sigma}(\gamma^s(1)).
\]
We set $H(s)$ the first term in RHS. 
We use the locally bi-Lipschitz homeomorphism $\D \times \R \to \widetilde{P}$,
\[
(z, t) \mapsto (z, \theta), \quad \theta = \varphi(z,t) = \int_0^t h(z, s) ds
\]
considered in Lemma \ref{lem:bilipschitz}
to change variables of the integration, which is justified thanks to Rademacher's theorem (see, e.g., \cite[Theorem 3.2.3]{Federer}).
It follows that
\begin{align*}
H(s)&= \int_{[0,1] \times [0,2\pi]} \left( \omega_1(\gamma^s(u),\theta) \frac {d\gamma^s_1}{du}(u)+ \omega_2(\gamma^s(u), \theta) \frac {d\gamma^s_2}{du}(u) \right) \frac{du d\theta}{2\pi}\\
&= \int_{[0,1] \times [0,2\pi]} \left( \frac{\partial\varphi}{\partial x_1}(\gamma^s(u), t) \frac{d \gamma^s_1}{du}(u) + \frac{\partial\varphi}{\partial x_2}(\gamma^s(u), t) \frac{d \gamma^s_2}{du}(u) \right) h(\gamma^s(u),t)\frac{du dt}{2\pi}.
\end{align*}
Harnack's inequality implies that 
\begin{equation}  \label{eq:cusp-harnack}
e^{-\delta^s_i} h(\gamma^s(0), t) \leq h(\gamma^s(u), t) \leq e^{\delta^s_i} h(\gamma^s(0), t)
\end{equation}
for any $u \in [0,1]$. So, we may approximate $H(s)$ by
\[
H_0(s) 
 := \int_{[0,1] \times [0,2\pi]} \left( \frac{\partial\varphi}{\partial x_1}(\gamma^s(u), t) \frac{d\gamma^s_1}{du}(u) + \frac{\partial\varphi}{\partial x_2}(\gamma^s(u), t) \frac{d\gamma^s_2}{du}(u) \right) h(\gamma^s(0),t)\frac{du dt}{2\pi}.
\]
We can rewrite
\begin{align*} 
H_0(s) 
& = \int_0^{2\pi}\frac{h(\gamma^s(0), t)dt}{2\pi} \int_0^1 \left( \frac{\partial\varphi}{\partial x_1}(\gamma^s(u), t) \frac{d\gamma^s_1}{du}(u) + \frac{\partial\varphi}{\partial x_2}(\gamma^s(u), t) \frac{d\gamma^s_2}{du}(u) \right) du\\
&
=  \int_0^{2\pi} \left( \varphi(\gamma^s(1), t) - \varphi(\gamma^s(0), t) \right) \frac{\mu'_{\gamma^s(0)}(dt)}{2\pi}, 
\end{align*}
where we used the notation $\mu'_z = h(z,t) \lambda(t)$.
Since $(\gamma^s(0), t), (\gamma^s(1), \tilde{\rho}'(c_i)(t)) \in \D \times \R$ are identified in $\Sigma \times_{\tilde{\rho}'} \R$, 
\[
\varphi(\gamma^s(1), t) - \tilde{\sigma}(\gamma^s(1))
= \varphi(\gamma^s(0), \tilde{\rho}'(c_i)^{-1}(t)) - \tilde{\sigma}(\gamma^s(0)).
\]
Thus, 
\begin{align}\label{eq:cusp-measure}
\begin{split}
& H_0(s) + \tilde{\sigma}(\gamma^s(0)) - \tilde{\sigma}(\gamma^s(1))\\
& = \int_0^{2\pi} \left( \varphi(\gamma^s(0), \tilde{\rho}'(c_i)^{-1}(t)) - \varphi(\gamma^s(0), t) \right) \frac{\mu'_{\gamma^s(0)}(dt)}{2\pi} \\
& = \int_0^{2\pi} \frac{\mu'_{\gamma^s(0)}(dt)}{2\pi} \int_t^{\tilde{\rho}'(c_i)^{-1}(t)} h(\gamma^s(0), u) du\\
& = \frac{1}{2\pi} \left( \mu'_{\gamma^s(0)} \otimes \mu'_{\gamma^s(0)}(E_i^+) - \mu'_{\gamma^s(0)} \otimes \mu'_{\gamma^s(0)}(E_i^-) \right),
\end{split}
\end{align}
where
\begin{align*}
E_i^+ &:= \{ (t, u) \in [0,2\pi] \times \R \mid t \leq u \leq \tilde{\rho}'(c_i)^{-1}(t) \},\\
E_i^- &:= \{ (t, u) \in [0,2\pi] \times \R \mid \tilde{\rho}'(c_i)^{-1}(t) \leq u \leq t \}.
\end{align*}

Since $\{ \mu'_{\gamma^s(0)}/2\pi \}_{s \gg 0}$ can be seen as a family of probability measures on $S^1$, we can extract a weakly convergent subsequence, say,
$\{ \mu'_{\gamma^{s_n}(0)}/2\pi \}_{n \in \N}$.
We denote its weak limit by $\nu$. We show that $\nu$ is $\tilde{\rho}'(c_i)$-invariant:
From the $\Gamma$-invariance of the harmonic measure as in \eqref{eq:fibermeas},
\[
\mu'_{\gamma^{s_n}(1)} = \tilde{\rho}'(c_i)_* \mu'_{\gamma^{s_n}(0)}.
\]
For any Borel subset $B$ in $\R$, \eqref{eq:cusp-harnack} implies that
\[
\lvert \mu'_{\gamma^{s_n}(1)}(B) -  \mu'_{\gamma^{s_n}(0)}(B)\rvert \leq \max\{ (e^{\delta^{s_n}_i} - 1), (1 - e^{-\delta^{s_n}_i}) \} \mu'_{\gamma^{s_n}(0)}(B).
\]
By letting $n \to \infty$, we conclude that $\nu(B) = \tilde{\rho}'(c_i)_* \nu(B)$.

Now we apply \eqref{eq:cusp-harnack} and \eqref{eq:cusp-measure} for $s = s_n$. Taking $n \to \infty$ yields
\begin{align*}
\lim_{n \to \infty} \tau(\widetilde\hol_{\overline\omega}(c_i^{s_n})) 
& = \frac{1}{2\pi} \lim_{n \to \infty} (H_0(s_n) + \tilde{\sigma}(\gamma^{s_n}(0)) - \tilde{\sigma}(\gamma^{s_n}(1)))\\
& =  \nu \otimes \nu(E_i^+) - \nu \otimes \nu(E_i^-).
\end{align*}
Since $\nu$ is $\tilde{\rho}'(c_i)$-invariant,
\[
\tau_\nu(\tilde{\rho}'(c_i)^{-1}) :=
\begin{cases}
\nu([t, \tilde{\rho}'(c_i)^{-1}(t)]) & \text{if $t \leq \tilde{\rho}'(c_i)^{-1}(t)$,}\\
-\nu(\tilde{\rho}'(c_i)^{-1}(t), t]) & \text{if $\tilde{\rho}'(c_i)^{-1}(t) \leq t$}\\
\end{cases}
\]
is independent of the choice of $t \in \R$. Hence,
\[
\lim_{n \to \infty} \tau(\widetilde\hol_{\overline\omega}(c_i^{s_n})) 
=  \int_0^{2\pi} \tau_\nu(\tilde{\rho}'(c_i)^{-1}) \nu(dt) = \tau_\nu(\tilde{\rho}'(c_i)^{-1}).
\]
This coincides with the translation number $-\tau(\tilde{\rho}'(c_i))$, which implies \eqref{eq:eachcusp}.  
The proof is completed.

\section{An application to rigidity theory}

Let us prove Corollary \ref{cor:mat} and Theorem \ref{thm:Poi} in this section. We first show the inequality \eqref{eq:MW}. If $e(\rho)=0$, then it is clear that the inequality \eqref{eq:MW} strictly holds. Thus let us assume that $e(\rho) \neq 0$. In that case, it is well known that $\rho (\Gamma)$ has no finite orbit in $S^1$. We can apply the construction in Section \ref{subsect:semiconj} and obtain a representation $\rho' \colon \Gamma \to \Homeo$ which is semiconjugate to $\rho$. We have $e(\rho) = e(\rho')$ as in Section \ref{subs:tr}.

We employ the continuous connection form constructed in Section \ref{subsect:curvature-est}. Then, from Theorem \ref{thm:GB}, we have
\begin{equation}\label{eq:MW-proof}
\lvert e(\rho')\rvert = \left\lvert \frac{1}{2\pi} \int_\Sigma K(z) \vol(dz)\right\rvert  \leq \frac{1}{2\pi} \int_\Sigma \lvert K(z) \rvert \vol(dz) \leq \frac{1}{2\pi}\vol(\Sigma) = \lvert e(\Sigma)\rvert,
\end{equation}
which shows the Milnor--Wood type inequality \eqref{eq:MW}.

Next, we show the rigidity statement of Corollary \ref{cor:mat}, 
and Theorem \ref{thm:Poi}.
For simplicity, we consider the case where $e(\rho) = e(\Sigma)$. 
The proof is similar in the other case where $e(\rho) = -e(\Sigma)$.
Since the equality holds in \eqref{eq:MW-proof}, we have $K(z) = -1$ for a.e.\ $z$. We fix $z_0 \in \Sigma$ where $K(z_0) = -1$. 

We use the notations in the proof of Claim \ref{inequality}. Recall that two inequalities, the isoperimetric inequality and Harnack's inequality, are involved in this proof:
\begin{align*}
K(z) & = -\frac{(1-\lvert z \rvert^2)^2}{4\pi} \int_0^{2\pi} \frac{1}{2}\left(-\frac{\partial \omega_1}{\partial \theta}(z, \theta) \omega_2(z, \theta) + \frac{\partial \omega_2}{\partial \theta}(z, \theta) \omega_1(z, \theta) \right) {d\theta}\\
&\geq -\frac{(1-\lvert z \rvert^2)^2}{4\pi} \cdot \frac{1}{4\pi} \left(\int_0^{2\pi} \sqrt{\left(\frac{\partial \log h}{\partial x_1}\right)^2 + \left(\frac{\partial \log h}{\partial x_2}\right)^2} d\theta\right)^2 \\
&= -\frac{1}{4\pi^2} \left(\int_0^{2\pi} \lvert d \log h \rvert_{\hyp} d\theta\right)^2\\
& \geq -1.
\end{align*}
From $K(z_0) = -1$, the equality holds in the isoperimetric inequality for the Lipschitz curve $(\omega_1(z_0,\theta), \omega_2(z_0,\theta))$ and this curve must be a round circle of radius $2/(1-\lvert z_0 \rvert^2)$ that is positively oriented. 
Also, the equality holds in Harnack's inequality for a.e.\ $\theta$ and the curve must have a constant speed.
Therefore, for some $\alpha(z_0) \in \R$, we have
\[
(h_{x_1}, h_{x_2})(z_0, \tau(z_0,\theta)) = \frac{2h(z_0, \tau(z_0,\theta))}{1- \lvert z_0 \rvert^2} (\cos (\theta + \alpha(z_0)), \sin(\theta +\alpha(z_0))),
\]
and hence
\[
\frac{h_{x_1}(z_0, t) + ih_{x_2}(z_0, t)}{\lvert h_{x_1}(z_0, t) + ih_{x_2}(z_0, t)\rvert} = e^{i( \varphi(z_0,t) + \alpha(z_0))}.
\]

From this and Lemma \ref{lem:harnack}, we have shown that the harmonic measure is of the form $h(z,t) \vol (z) \lambda (t)$, where
\begin{equation}\label{eq:harmonicmeas}
h(z,t) =  \frac{1-\lvert z \rvert^2}{\lvert m(e^{it})-z \rvert^2}
\end{equation}
for a.e.\ $t$ with
\begin{align*}
m(e^{it}) 
:= \frac{e^{i(\varphi(z_0,t) + \alpha(z_0))}+ z_0}{1 + \overline{z_0}e^{i(\varphi(z_0,t) + \alpha(z_0))}} \in S^1.
\end{align*}
Note that, by \eqref{eq:harmonicmeas}, for a.e.\ $t$, $m(e^{it})$ is the unique point in $S^1$ such that $h(z,t) \to +\infty$ as $z \to m(e^{it})$, which is independent of the choice of $z_0$.

The following claim finishes the proof of Corollary \ref{cor:mat} since $\rho$ is semiconjugate to $\rho'$ by $\psi$.

\begin{claim}\label{claim:equiv}
$m \colon S^1 \to S^1$ is a $(\rho', \rho_0)$-equivariant orientation-preserving homeomorphism, which conjugates $\rho'$ to $\rho_0$.
\end{claim}

\begin{proof}
By Lemma \ref{lem:bilipschitz}, $\varphi$ is an orientation-preserving homeomorphism, and hence so is $m$. 

Let us prove the equivariance of $m$, which follows from the invariance of $\tilde{\mu}'$ under the $\G$-action on $\D \times \R$ as in the sequel: Since the harmonic measure $\widetilde{\mu}' = h(z,t)\vol(z)\lambda(t)$ on $\D \times \R$ is the pull-back of a measure on $M'$, 
we have, for every $\gamma \in \Gamma$, $\gamma_*\widetilde{\mu}' = \widetilde{\mu}'$. It follows that 
\begin{equation}\label{eq:equiv}
h(\rho_0(\gamma)z, \tilde\rho'(\gamma)(t)) \vol(z) \rho'(\gamma)_*\lambda(t) = h(z,t)\vol(z)\lambda(t).
\end{equation}
Then, for a.e.\ $t$, $h(\rho_0(\gamma)z, \tilde\rho'(\gamma)(t))$ and $h(z,t)$ differ only up to a constant which may depend on $t$. As $z \to m(e^{it})$, the right hand side of \eqref{eq:equiv} diverges to $+\infty$, and hence so does the left hand side. Since $m(e^{it})$ is the unique point in $S^1$ such that $h(z,t) \to +\infty$ as $z \to m(e^{it})$ for a.e.\ $t$ by \eqref{eq:harmonicmeas}, it follows that $\rho_0(\gamma) \circ m = m \circ \rho'(\gamma)$.
\end{proof}

Now consider $\mathfrak{m} := m \circ \psi \colon S^1 \to S^1$. 
The map $\mathfrak{m}$ is continuous, monotone, of degree one and $(\rho, \rho_0)$-equivariant from Claims \ref{newholonomy} and \ref{claim:equiv}  
giving a semiconjugacy from $\rho$ to $\rho_0$.
The equation \eqref{eq:harmonicmeas} completely describes the leafwise harmonic function associated with $\mu'$ on $M \times_{\rho'} S^1$ and $\Psi$.
From $\mu'_z = \psi_* \mu_z$, it follows that
\begin{align*}
\mu_z (\psi^{-1}(B)) & = \int_B \frac{1-\lvert z \rvert^2}{\lvert m(e^{it})-z \rvert^2} \lambda(dt)\\
& = \int_B \frac{1-\lvert z \rvert^2}{\lvert m(e^{it})-z \rvert^2} (\psi_* \mu_0)(dt)\\
&= \int_{\psi^{-1}(B)} \frac{1-\lvert z \rvert^2}{\lvert \m(e^{it})-z \rvert^2} q(0,t)\nu(dt)
\end{align*}
for any Borel set $B$ on $S^1$.
Since, for any interval $I \subset \R$,  $\tilde{\psi}^{-1}(\tilde{\psi}(I)) \setminus I$ is $\nu$-null set, we yield
\[
\mu_z (B) = \int_{B} \frac{1-\lvert z \rvert^2}{\lvert \m(e^{it})-z \rvert^2} q(0,t)\nu(dt)
\]
for any open set, hence, any Borel set $B$ on $S^1$. 
We therefore obtained the description of 
the leafwise harmonic function associated with $\mu$ on $M \times_{\rho} S^1$,
\[
q(z,t) =  q(0,t)\frac{1-\lvert z \rvert^2}{\lvert \m(e^{it})-z \rvert^2}
\]
for $\nu$-a.e.\ $t$.
Replacing $\nu$ with $q(0,t) \nu$, we complete the proof of Theorem \ref{thm:Poi}.

\end{document}